\documentclass{article}
\usepackage[utf8]{inputenc}
\usepackage{amsthm,amssymb, amsmath,enumerate}
\usepackage{graphicx}

\newcommand{\E}[1]{\mathbb{E}\left(#1\right)}
\newcommand{\Edg}[1]{\widetilde{\mathbb{E}}\left(#1\right)}
\newcommand{\Pof}[1]{\mathbb{P}\left(#1\right)}

\renewcommand{\P}{\mathbb{P}}

\renewcommand{\H}{\mathbb{H}}
\newcommand{\C}{\mathbb{C}}

\newcommand{\R}{\mathbb{R}}
\newcommand{\Z}{\mathbb{Z}}

\DeclareMathOperator{\dist}{dist}
\DeclareMathOperator{\acosh}{acosh}

\newtheorem{theorem}{Theorem}
\newtheorem{remark}{Remark}
\newtheorem{proposition}{Proposition}
\newtheorem{lemma}{Lemma}
\newtheorem{corollary}{Corollary}
\newtheorem*{definition*}{Definition}
\newtheorem*{theorem*}{Theorem}
\newtheorem{theoremalph}{Theorem}

\newtheorem{question}{Question}

\usepackage{hyperref}
\usepackage[alphabetic,initials]{amsrefs}

\usepackage{parskip}
\makeatletter
\def\thm@space@setup{%
  \thm@preskip=\parskip \thm@postskip=0pt
}
\makeatother

\usepackage{etoolbox}
 
\makeatletter
\patchcmd{\maketitle}{\@fnsymbol}{\@alph}{}{}  % Footnote numbers from symbols to small letters
\makeatother

\usepackage{dsfont}

\title{On the speed of distance stationary sequences}
\author{
   Matías Carrasco \thanks{Inst.\ Matemática y Estadística, Universidad de la Rep\'ublica, \href{mailto:mcarrasco@fing.edu.uy}{mcarrasco@fing.edu.uy}}
   \and
   Pablo Lessa \thanks{Inst.\ Matemática y Estadística, Universidad de la Rep\'ublica, \href{mailto:plessa@fing.edu.uy}{plessa@fing.edu.uy}}
   \and
   Elliot Paquette \thanks{Department of Mathematics, The Ohio State University, \href{paquette.30@osu.edu}{paquette.30@osu.edu}}
}

\begin{document}
\maketitle

\begin{abstract}
We prove a formula for the speed of distance stationary random sequences generalizing the law of large numbers of Karlsson and Ledrappier. A particular case is the classical formula for the largest Lyapunov exponent of i.i.d.\ matrix products, but our result has applications in various different contexts. In many situations it gives a method to estimate the speed, and in others it allows to obtain results of dimension drop for escape measures related to random walks. We show applications to stationary reversible random trees with conductances, Bernoulli bond percolation of Cayley graphs, and random walks on cocompact Fuchsian groups.
 
 
\medskip \noindent {\bf AMS2010:} Primary 60Gxx, 60Dxx, 51M10, 05C81. 
 \end{abstract}

\section{Introduction}

The law of large numbers of Karlsson and Ledrappier \cite{karlsson-ledrappier}, following previous works by Kaimanovich \cite{kaimanovich87} and Karlsson and Margulis \cite{karlsson-margulis}, has numerous applications including the Osledets theorem and other results of geodesic tracking, applicable for example to actions on Teichmüller spaces \cite{karlsson2004linear,karlsson-goezel}.

In this article we generalize this theorem obtaining, in addition to an escape direction, an integral formula for the speed of a random walk. In the context of i.i.d.\ matrix products, this implies the Furstenberg formula for the Lyapunov exponent which is not obtained directly from the result of Karlsson and Ledrappier.

Our formula holds for a vast class of random walks that we call \emph{distance stationary}. Let $M$ be a complete and separable metric space, a random sequence $(x_n)_{n \in \Z}$ of points in $M$ is said to be distance stationary if the distribution of $(d(x_m,x_n))_{m,n \in \Z}$ coincides with that of $(d(x_{m+1},x_{n+1}))_{m,n \in \Z}$. If in addition the mean distance of the first step is finite, i.e\ $\E{d(x_0,x_1)} < +\infty$, then by Kingman's subadditive ergodic theorem the random limit
\[\ell = \lim\limits_{n \to +\infty}\frac{1}{n}d(x_0,x_n)\]
exists almost surely and in $L^1$. We call this limit the \emph{speed} or \emph{linear drift} of $(x_n)_{n \in \Z}$.

\begin{theoremalph}[see Theorem \ref{furstenbergtheorem}]\label{thm:A}
  If $(x_n)$ is distance stationary and the mean distance of the first step is finite, then there exists a random horofunction $\xi$ such that \(\E{\ell} = -\E{\xi(x_1)-\xi(x_0)}\) and so that $\frac{\xi(x_n)}{n}$ converges almost surely and in $\text{L}^1$ to $\ell$ as $n\to\infty.$
\end{theoremalph}
\noindent Note that we do not assume that $M$ is \emph{proper}, and so we partially answer a question of \cite{karlsson-ledrappier}.

Recall that the horofunction compactification of $M$ is the space $\widehat{M}$ obtained as the closure of the functions of the form $\xi_x(y) = d(x,o) - d(x,y)$ in the topology of uniform convergence on compact sets. Here $o$ is a base point in $M$ and usually our process starts at $x_0=o$. If in addition the speed is deterministic the formula simplifies to $\ell=-\E{\xi(x_1)}$.

The horofunction boundary of $M$ is, abusing notation slightly, the space $\widehat{M} \setminus M$. In all of our applications the random horofunction $\xi$ is indeed a boundary horofunction, see Proposition \ref{corollaryboundary}. It is obtained as a limit of random horofuntions involving the steps $x_n$ for negative values of $n$, and even though it is quite difficult to obtain non trivial information about its distribution, in may applications it is enough to know that it is independent of the first step $x_1$. We refer the reader to Theorem \ref{furstenbergtheorem} in Section \ref{sec:furstenberg} for a detailed statement of Theorem \ref{thm:A}, including the measurability properties of $\xi$. We also postpone the proof of the theorem to that section.

Our result has applications in various contexts not related to matrix products. In many situations it gives a method to estimate the speed, and in others it allows one to obtain results of dimension drop for escape measures related to random walks. In the rest of this introduction we outline some of these applications.

In Section \ref{sec:trees}, we obtain from our theorem a formula for the speed of the simple random walk in a reversible stationary random tree with conductances or weights (see Lemma \ref{electricspeedlemma}). The formula is comparable to that of Aldous and Lyons given in \cite[Theorem 4.9]{aldous-lyons}. Notice that it does not require the mean degree of the root to be finite.  See also \cite[Section 6]{benjamini-unimodular} for a discussion on this point.

\begin{theoremalph}[see Lemma \ref{electricspeedlemma}]
In a reversible random tree with conductances, if the simple random walk $(x_n)$ is almost surely transient and has deterministic speed, then 
\[\ell = \E{\frac{AC}{AB+AC+BC}}\]
where \(B\) is the conductance of the edge \(e\) between \(x_0\) and \(x_1\), \(A\) is the effective conductance between \(x_0\) and infinity after removing \(e\), and \(C\) is the effective conductance between \(x_1\) and infinity after removing \(e\).
\end{theoremalph}

As an application of this formula we show that recurrence and zero speed are equivalent for trees with conductances with more than one end, assuming the inverse of the root weight is integrable.

\begin{theoremalph}[see Theorem \ref{electricaltheorem}]
Suppose \((T,o)\) is a stationary reversible random tree with conductances such that the inverse of the root conductance is integrable. If the simple random walk on \((T,o)\) is almost surely transient and has zero speed, then \((T,o)\) has one end almost surely.
\end{theoremalph}

In the case without weights (or more precisely all weights are equal to 1) this result was obtained by Curien and appears in his notes \cite[Theorem 4.1]{curien} where it is mentioned that there is no published reference. A related formula was proved for Galton-Watson trees with conductances in \cite{gantert-muller-popov2012}.  

We give an example of a reversible tree with one end where the simple random walk is transient but has zero speed in section \ref{oneendexample}.  The example, a weighted Canopy tree, is essentially the same as the ones given in \cite[Section 5.3]{benjamini-curien} and \cite[Section 1.3]{gurel-nachmias}.

In Section \ref{sec:percolation} we focus on a particular application to the simple random walk on Bernoulli percolation clusters of Cayley graphs as it is defined in \cite{benjamini-lyons-schramm1999}: given a Cayley graph $G$ of a group with respect to a finite and symmetric generating set $F$, rooted at the identity element, and a probability $p\in [0, 1]$, we consider the connected component of the root $G_p$ of the random subgraph of $G$ formed by deleting each edge independently with probability $1-p$. It was shown in \cite[Lemma 4.2]{benjamini-lyons-schramm1999} that the speed \(\ell_p\) of the simple random walk on \(G_p\) conditioned on \(G_p\) being infinite is deterministic.

We consider on \(G\) a distance which is left invariant but is not necessarily the graph distance. Under the hypothesis that the sum of boundary horofunctions in the generators $F$ are always negative
\begin{equation}\label{eq:hyperbolicity}
\sum\limits_{x \in F}\xi(x) < 0,\quad (\text{for all }\xi\in \hat{G}\setminus G)
\end{equation}
we obtain:

\begin{theoremalph}[see Theorem \ref{bernoullipercolationtheorem}]
In the context above, suppose that \eqref{eq:hyperbolicity} holds for all boundary horofunctions \(\xi\). Then \(\ell_p > 0\) for all \(p\) sufficiently close to \(1\). Moreover, for all \(p \neq 0\) one has
\[\ell_p \ge \frac{1}{|F|}\sum\limits_{x \in F}\dist(o,x) - \frac{1}{p|F|}\max\limits_{\xi}\sum\limits_{x \in F}f_\xi(x),\]
where the maximum is over all boundary horofunctions, and $f_\xi(x)$ is the function $\xi(x)+\dist(o,x)$.
\end{theoremalph}
The condition expressed in equation \eqref{eq:hyperbolicity} is a form of hyperbolicity for the group $G$ (see Remark \ref{rmk} in Section \ref{sec:percolation}). In this sense, our result can be compared to \cite[Theorem 1.3]{benjamini-lyons-schramm1999} where it is proved that if the Cayley graph is non-amenable, then the simple random walk on $G_p$ has positive speed, and also to \cite[Theorem 1.1]{chen2004} where it is proved that if $G$ has a
positive anchored expansion constant, then so does every infinite cluster for $p$ sufficiently close to 1 (see also \cite{chen2003}). The point to be highlighted from our result is that it provides a lower bound for $\ell_p$.

As an illustration, we can apply the previous result to estimate the speed of the simple random on a percolation cluster of a hyperbolic tiling $T_{P,Q}$ by regular $P$-gons with interior angle $2\pi/Q$, i.e.\ $Q$ polygons meet at a given vertex.

\begin{theoremalph}[see Theorem \ref{tilingtheorem}]
In the above context \[\ell_p(P,Q) \ge 2\log(Q) - \frac{1}{p}O(\log(\log(Q)))\] when \(Q \to +\infty\) where the right hand side is independent of \(P\).
\end{theoremalph}

Letting $p$ to be equal to 1, and combining the previous estimate with the relationship between entropy, dimension and speed (see for example Tanaka \cite{tanaka}, Hochman and Solomyak \cite{hochman-solomyak}, and Ledrappier \cite{ledrappier1984}) the following dimension drop result for the escape measure is obtained:

\begin{theoremalph}[see Theorem \ref{dimdrop}]
Let \(\nu_{P,Q}\) be the harmonic measure of the simple random walk on $T_{P,Q}$, and let \(\dim(\nu_{P,Q})\) be its Hausdorff dimension. Then one has
\[\limsup\limits_{Q \to +\infty}\dim(\nu_{P,Q}) \le \frac{1}{2}\]
uniformly in \(P\).
\end{theoremalph}

This was first proved in \cite{furstenbergformulaoldversion} as far as the authors are aware. In particular the harmonic measure $\nu_{P,Q}$ is singular with respect to the Lebesgue measure on the circle. This result was re-obtained independently in a recent work of Petr Kosenko \cite{kosenko}. His result covers all but finitely many cases for the number of sides and polygons per vertex \((P,Q)\), while we estimate the dimension of the boundary measure only for large \(Q\).

We hope that these methods are applicable in other situations, such as the study of the speed and harmonic measure of hyperbolic Poisson-Delaunay walks (see \cite{benjamini-paquette-pfeffer}) and possibly other models such as the Planar Stochastic Hyperbolic Infinite Triangulations (PSHT, see \cite{curien2014}).

\section{Application to random trees}\label{sec:trees}

In this section we will apply Theorem \ref{furstenbergtheorem} to the random walk on reversible random trees with random conductances.

Recall that a random rooted graph \((G,o)\) is stationary if its distribution is invariant under re-rooting by a simple random walk (see \cite{benjamini-curien}).  That is, if  \(o=x_0,x_1,\ldots\) is a simple random walk starting at \(o\) on \(G\), then \((G,x_1)\) has the same distribution as \((G,o)\).

A stationary random graph is said to be reversible if the doubly rooted graphs \((G,x_0,x_1)\) and \((G,x_1,x_0)\) have the same distribution.  A common source of reversible stationary graphs are unimodular random graphs whose distribution has been biased by a density proportional to the degree of the root vertex.

Notice that these definitions trivially extend to graphs with conductances, that is with positive weights associated to each edge.   In this case the simple random walk refers to the Markov chain where the transition probabilities from one vertex to another is proportional to the conductance of the corresponding edge.

The speed of a simple random walk \(x_0=o,x_1,\ldots\) starting at the root of a stationary random graph \((G,o)\) is defined as
\[\ell = \lim\limits_{n \to +\infty}\frac{1}{n}\dist(x_0,x_n)\]
which exists almost surely and in mean by the subadditive ergodic theorem.  We say speed is deterministic if it is almost surely equal to a constant.

We will prove that trees with zero speed are recurrent unless they have exactly one end.   We also give an example of a weighted reversible random tree where the random walk is transient but has zero speed.  Notice that all unweighted one ended trees are recurrent.  In the case of unweighted trees the result below was proved in \cite[Theorem 4.1]{curien}.

\begin{theorem}[Zero speed on reversible trees with conductances.]\label{electricaltheorem}
Let \((T,o)\) be a stationary reversible random tree with conductances such that \(\E{\frac{1}{c(o)}} < +\infty\) where \(c(o) = \sum\limits_{x}c(o,x)\) and \(c(x,y)\) denotes the conductance of the edge between vertices \(x\) and \(y\).

If the simple random walk on \((T,o)\) almost surely has zero speed and is transient, then \((T,o)\) has one end almost surely.
\end{theorem}

\subsection{Weighted Canopy trees\label{oneendexample}}

We will now construct a reversible random tree with one end on which the simple random walk is transient but has zero speed almost surely. Similar examples are constructed in \cite[Section 5.3]{benjamini-curien} and \cite[Section 1.3]{gurel-nachmias}.

Consider the binary Canopy tree \(C\) (see \cite{aizenman-warzel}) which is constructed as follows:
\begin{enumerate}
 \item Start with countably many leaves, which are the vertices at level \(0\).
 \item For each \(n=0,1,2,\ldots\), group the vertices at level \(n\) into pairs and join each pair to a new vertex at level \(n+1\).
\end{enumerate}

Suppose that the edges joining a vertex at level \(n\) with a vertex at level \(n+1\) have conductance \(c_n = \lambda^n\) where \(1 < \lambda < 2\).

Since \(\sum\limits_{n = 0}^{+\infty}\frac{1}{c_n} < +\infty\) the simple random walk on \(C\) is transient.

Let \(x_0,x_1,\ldots\) be a simple random walk on \(C\) and \(l_n\) be the level of the vertex visited at the \(n\)-th step of the random walk one has
\[\Pof{l_{n+2}= k+1|l_n = k} = \frac{\lambda^{n}}{2\lambda^{n-1} + \lambda^n} = \frac{\lambda}{2+\lambda} < \frac{1}{2},\]
for all \(k = 1,2,3,\ldots\).

It follows that \(l_n, n = 0,1,2,\ldots\) is positively recurrent.

Letting \(p_k\) be the frequence with which \(l_n = k\), and rooting \(C\) at a random vertex \(o\) which has level \(k\) with probability \(p_k\), one obtains a reversible stationary random tree with conductances on which the simple random walk is almost surely transient.

Notice that there is essentially a unique boundary horofunction on \(C\), the caveat being that changing the root modifies this horofunction by an additive constant since we demand \(\xi(o) = 0\) for all horofunctions.  Up to an additive integer constant it is simply the function assigning to each vertex its level.

Applying Theorem \ref{furstenbergtheorem} and Proposition \ref{corollaryboundary} yields that the speed of the simple random walk on \(C\) is the expected value of the increment in level on the first step.   Since \(l_n\) is positively recurrent, this is zero.

Hence, we have constructed reversible random trees with conductances where the speed of the simple random walk is zero but the random walk is almost surely transient.

\subsection{Speed formulas}

The first step in the proof of Theorem \ref{electricaltheorem} is a speed formula which we obtain from Theorem \ref{furstenbergtheorem}.   It can be compared to the formulas given in \cite{lyons-permantle-peres1995} for Galton-Watson trees, which is essentially Lemma \ref{treehorofunctionlemma} below, and in \cite{gantert-muller-popov2012} for Galton-Watson trees with random conductances.

\begin{lemma}\label{electricspeedlemma}
Let \((T,o)\) be a reversible random tree with conductances, and \(x_0=o,x_1,x_2,\ldots\) a simple random walk on \((T,o)\).   Suppose the simple random walk is almost surely transient and has deterministic speed \(\ell\).

Then
\[\ell = \E{\frac{AC}{AB+AC+BC}}\]
where \(B\) is the conductance of the edge \(e\) between \(x_0\) and \(x_1\), \(A\) is the effective conductance between \(x_0\) and infinity after removing \(e\), and \(C\) is the effective conductance between \(x_1\) and infinity after removing \(e\).
\end{lemma}

Since Theorem \ref{furstenbergtheorem} requires a distance stationary sequence of a single complete separable metric space \(X\), we need to embed the reversible random tree \((T,o)\) into such a space.

For this purpose let \((X,o)\) be a rooted regular tree with root \(o\) where all vertices have countably many neighbors.  We assume furthermore that a total linear order with a minimal element is given for the children of each vertex \(x \in X\).

\begin{lemma}\label{treeembeddinglemma}
In the context of Theorem \ref{electricaltheorem}, there exists a random subtree \((T',o)\) of \((X,o)\) with the same distribution as \((T,o)\), and the additional property that almost surely for all \(x \in X\), if \(x \in T'\) has \(k\) children in \(T'\) then these are the first \(k\) children of \(x\) in the established ordering.
\end{lemma}
\begin{proof}
Beginning with the identification of the root of \((T,o)\) with \(o \in X\), we may construct the required \(T'\) by considering countably many independent random walks on \((T,o)\) and each time a new child of a vertex is explored, mapping it to the lowest available child of the corresponding node in \((X,o)\).  Since the walks will almost surely visit all vertices of \((T,o)\), this gives a measurable random subgraph \(T'\) of \((X,o)\) which is isomorphic to \((T,o)\) almost surley.
\end{proof}

We assume from now on that \((T,o)\) is a random subtree of \((X,o)\) with random conductances as given by the previous lemma.

\begin{lemma}\label{treehorofunctionlemma}
In the context of Lemma \ref{electricspeedlemma}, let \(x_0'=o,x_1',\ldots\) be a simple random walk starting at \(o\) independent from \(x_0,x_1,\ldots\) and let \(\xi\) be the boundary horofunction correponding to the loop erased path obtained from \((x_n')\).  Then \(\ell = -\E{\xi(x_1)}\).
\end{lemma}
\begin{proof}
By reversibility, conditioned on \((T,x_1)\), one has that \(x_0\) has the distribution of the first step of a simple random walk on \(T\) starting at \(x_1\).

Letting \(x_{-n}, n = 0,1,2,\ldots\) be an simple random walk starting from \(o\) and independent from \(x_n,n=0,1,2,\ldots\), this implies that \((T,x_0, (x_n)_{n \in Z})\) has the same distribution as \((T,x_1,(x_{n+1})_{n \in \Z})\) in the space of rooted trees with bi-infinite paths.  In particular the sequence \((x_n)_{n \in \Z}\) is distance stationary in \((X,o)\).

By hypothesis, speed is deterministic and therefore \(\ell = -\E{\xi(x_1)}\) where \(\xi\) is the random horofunction on \(X\) given by Theorem \ref{furstenbergtheorem}. Notice that \(\E{\dist(x_0,x_1)} \le 1\) so the integrability condition is satisfied.

Notice also that since by hypothesis the simple random walk is transient, the sequence of probabilities
\[\mu_n = \frac{1}{n}\sum\limits_{i = 1}^n \delta_{\xi_{x_{-i}}}\]
converges almost surely to the Dirac delta on the unique horofunction which increases by \(1\) at each step of the path obtained by erasing all loops from \(x_0,x_{-1},x_{-2},\ldots\).
\end{proof}

\begin{proof}[Proof of Lemma \ref{electricspeedlemma}]
Compactify \(L\) with a single point \(-\infty\) at infinity and \(R\) with a single point \(+\infty\) at infinity.

With this notation \(A\) is the effective conductance between \(x_0\) and \(-\infty\) in \(L\), \(B\) is the conductance of the edge between \(x_0\) and \(x_1\), and \(C\) is the effective conductance between \(x_1\) and \(+\infty\) in \(R\).
Notice that since the simple random walk is transient \(\max(A,C) \neq 0\) almost surely.

From Lemma \ref{treehorofunctionlemma} one has \(\ell = -\E{\xi(x_1)}\).    Notice that \(\xi(x_1) = 1\) if and only if \(x_{-n} \in R\) for all \(n\) large enough, and \(\xi(x_1) = -1\) if and only if \(x_{-n} \in L\) for all \(n\) large enough.

In the network with 4 vertices \(-\infty,x_0,x_1,+\infty\) where the edges have conductances \(A,B,C\) respectively, notice that the second and third edges combined have effective conductance \(\frac{BC}{B+C}\).

From this it follows that
\[-\E{\xi(x_1)|(T,x_0,x_1)} = -\left(\frac{\frac{BC}{B+C}}{A + \frac{BC}{B+C}}-\frac{A}{A + \frac{BC}{B+C}}\right) = \frac{AB+AC-BC}{AB+AC+BC}.\]

Taking expected value and noticing that by reversibility 
\[\frac{AB}{AB + AC + BC}\quad \text{and} \quad\frac{BC}{AB + AC + BC}\]
have the same distribution one obtains
\begin{align*}
\ell &= -\E{\E{\xi(x_1)|(T,x_0,x_1)}}\\ & = \E{\frac{AB+AC-BC}{AB+AC+BC}} = \E{\frac{AC}{AB+AC+BC}},\end{align*}
as claimed.
\end{proof}

\subsection{Transient trees with zero speed have one end}

We will complete the proof of Theorem \ref{electricaltheorem} by starting from Lemma \ref{electricspeedlemma}, and showing that an almost surely transient reversible tree with \(\ell = 0\) has one end.

For this purpose we define
\[\Edg{*} = \frac{\E{*}}{\E{\frac{1}{c(o)}}}.\]

Notice that one has the following version of the mass-transport principle:
\[\Edg{\sum\limits_{x} f(T,o,x)c(o,x)} = \Edg{\sum\limits_{x} f(T,x,o)c(o,x)},\]
where \(f\) is any positive function on doubly rooted trees with conductances.

We denote by \(\deg(x)\) the number of neighbors of a vertex \(x\) in a graph.

\begin{lemma}
In the context of Lemma \ref{electricspeedlemma}, if \(\ell = 0\) then \(\Edg{\deg(o)} = 2\).
\end{lemma}
\begin{proof}
In view of Lemma \ref{electricspeedlemma}, since \(\ell = 0\), one has that almost surely exactly one of the two trees \(L\) and \(R\) is recurrent and the other is transient.

By stationarity, it follows that almost surely for all \(n\) removing the edge between \(x_n\) and \(x_{n+1}\) from \(T\) leaves one recurrent tree and one transient subtree.
Hence, almost surely there is an infinite non-backing path \(y_0 = o,y_1,\ldots\) such that removing the edge \(e_n\) between \(y_n\) and \(y_{n+1}\) splits \(T\) into a transient and a recurrent tree for all \(n\), where the recurrent tree is the component containing \(y_0 = o\).

Let \(f(T,o,x)\) be \(1/c(o,x)\) if \(x\) is the unique neighbor of \(o\) in the transitive component of the graph obtained after removing \(o\).  Applying the mass-transport principle one has
\[1 = \Edg{\sum\limits_{x}f(T,o,x)c(o,x)} = \Edg{\sum\limits_{x}f(T,x,o)c(o,x)} = \Edg{\deg(o)-1}.\]
\end{proof}

We now repeat the argument from \cite[Theorem 13]{curien} which shows that having expected degree \(2\) has strong topological consecuences.

\begin{lemma}
In the context of Lemma \ref{electricspeedlemma}, if \(\ell = 0\) then almost surely \(T\) has one or two ends.
\end{lemma}
\begin{proof}
Let \(f(T,x,y) = 1/c(x,y)\) if removing the edge between \(x\) and \(y\) leaves \(x\) in a finite component.

Deterministically one has
\[\deg(o) + \sum\limits_{x}f(T,o,x)c(o,x) - \sum\limits_{x}f(T,x,o)c(o,x) \ge 2.\]
But applying \(\Edg{*}\), by hypothesis both sides of the inequality have the same expected value.  It follows that both sides are equal almost surely.

Notice that if there are three infinite trees connected at \(o\), then the left hand side is at least 3.  Hence, almost surely there are at most two edges connecting \(o\) to infinite components.

By stationarity this holds almost surely at all vertices visited by the walk.   Since all vertices have positive probability of being visited, this holds almost surely at all vertices of  \(T\).  Hence, \(T\) almost surely has one or two ends as claimed.
\end{proof}

\begin{lemma}
In the context of Lemma \ref{electricspeedlemma}, if \(\ell = 0\) then \((T,o)\) has one end almost surely.
\end{lemma}
\begin{proof}
Suppose, for the sake of contradiction, that \(\ell = 0\) but \((T,o)\) has more than one end with positive probability.  We have shown that \((T,o)\) almost surely has at most two ends.   Conditioning we may assume from now on that \((T,o)\) has exactly two ends almost surley.

Let \(x_0=o,x_1,x_2,\ldots\) be a simple random walk on \((T,o)\) starting at \(o\).

Since \((T,o)\) almost surely has two ends there is a unique geodesic path joining the two.

Splitting this geodesic into two rays, it follows from transience of the random walk that the sum of inverses of the conductances along one of the two rays must be finite, and that this is the end to which all simple random walks on \((T,o)\) converge.

Given \(\epsilon > 0\), consider the event \(A_{n,\epsilon}\) that \(x_n\) is on the geodesic joining the two ends of \(T\) and all edges at \(x_n\) have conductance at least \(\epsilon\).

Notice that this sequence of events is stationary and for some small enough \(\epsilon > 0\) each \(A_{n,\epsilon}\) has positive probability. 

It follows from Birkhoff's ergodic theorem that with positive probability
\[\lim\limits_{n \to +\infty}\frac{1}{n}\sum\limits_{k = 0}^{n-1}\mathds{1}_{A_{k,\epsilon}} > 0.\]

But this contradicts the fact that the sums of inverse conductances converges along the ray pointing towards the end to which the simple random walk converges.
\end{proof}

\section{Application to Bernoulli percolation clusters}\label{sec:percolation}

In this section we will show how Theorem \ref{furstenbergtheorem} allows one to obtain criteria for positive speed of the simple random walk on infinite percolation clusters of certain Cayley graphs.   We will also obtain estimates on the speed in certain cases.  The results belong to the same context as those of \cite{benjamini-lyons-schramm1999} but are new as far as the authors are aware.

\subsection{Bernoulli percolation clusters on Cayley graphs}

Suppose \(G\) is a finitely generated group and \(F\) is a finite symmetric generator (i.e. \(g \in F\) implies \(g^{-1} \in F\)).   Slighly abusing notation denote by \((G,o)\) the Cayley graph of \(G\) with respect to \(F\) rooted at the identity element \(o\).  Here the vertices are elements of \(G\) and two elements \(x,y\) are connected by a single undirected edge if \(x = yg\) for some \(g \in F\).

We consider on \((G,o)\) a distance which is left \(G\)-invariant but is not necessarily the graph distance.   On all subgraphs of \(G\) we will consider the restriction of this distance function, in particular the boundary horofunctions on a subgraph are a subset of those of \(G\).   It will also be useful to consider for each boundary horofunction \(\xi\) the function \(f_\xi(x) = \dist(o,x) + \xi(x)\) which is always non-negative.

We now describe the \(p\)-Bernoulli percolation \(G_p\) on \(G\) as defined in \cite{benjamini-lyons-schramm1999}.

For each edge \(e\) in \(G\) consider an independent random variable \(u(e)\) uniformly distributed in \([0,1]\).

For each \(p \in [0,1]\) we let \((G_p,o)\) be the connected component of \(o\) in the subgraph of \(G\) consisting of edges with \(u(e) < p\).

It is well known (see for example \cite[Theorem 2.1]{haggstrom-peres}) that \((G_p,o)\) is unimodular, and hence stationary and reversible if its distribution is biased by a density proportional to the number of neighbors of \(o\).   This extends trivially to the extra data of the restriction of the ambient distance coming from \(G\) to \(G_p\).

It was shown in \cite[Lemma 4.2]{benjamini-lyons-schramm1999} that the speed \(\ell_p\) of the simple random walk on \(G_p\) conditioned on \(G_p\) being infinite is deterministic.

\begin{theorem}[Speed estimate for Bernoulli percolation clusters]\label{bernoullipercolationtheorem}
In the context above suppose that \(\sum\limits_{x \in F}\xi(x) < 0\) for all boundary horofunctions \(\xi\).

Then \(\ell_p > 0\) for all \(p\) sufficiently close to \(1\), and for all \(p \neq 0\) one has
\[\ell_p \ge \frac{1}{|F|}\sum\limits_{x \in F}\dist(o,x) - \frac{1}{p|F|}\max\limits_{\xi}\sum\limits_{x \in F}f_\xi(x),\]
where the maximum is over all boundary horofunctions.
\end{theorem}
\begin{proof}
 Let \(y_0=o,y_1,\ldots\) and \(y_0'=o,y_1',\ldots\) be two independent random walks on \(G_p\), both starting at \(o\), and let \(x_n = y_n\) if \(n \ge 0\), while \(x_n = y_{-n}'\) if \(n \le 0\).  Let \(\deg(o)\) denote the degree of \(o\) in \(G_p\).
 
 Notice that \((x_n)_{n \in \Z}\) is distance stationary under the probability biased by \(\deg(o)\), and let \(\xi'\) be the horofunction given by Theorem \ref{furstenbergtheorem}.
 
 The simple random walk on an infinite connected (unweighted) graph cannot be strongly recurrent. In effect, the probability that the walk is in any fixed finite set goes to zero as \(n\) goes to infinity.  Therefore, by Proposition \ref{corollaryboundary} one has that \(\xi'\) is a boundary horofunction.
 
 For each boundary horofunction \(\xi\) let \(f_\xi(x) = \xi(x) + \dist(o,x)\) and notice that \(f_\xi \ge 0\).
 
 Since \(\ell_p\) is deterministic, one obtains 
 \begin{align*}
 \frac{1}{|F|}\sum\limits_{x \in F}\dist(o,x) - \ell_p &= \E{\deg(o)}^{-1}\E{(\dist(x_0,x_1) + \xi'(x_1))\deg(o)} 
 \\ &= \E{\deg(o)}^{-1}\E{\sum\limits_{x \sim o}f_{\xi'}(x)} \le \frac{1}{p|F|}\E{\sum\limits_{x \in F}f_{\xi'}(x)}
 \\ &\le \frac{1}{p|F|}\max\limits_{\xi}\sum\limits_{x \in F}f_\xi(x).
 \end{align*}
\end{proof}

We will give an example where the above result can be applied for a Fuchsian group \(G\) using the distance coming from the hyperbolic plane.   However, the following question seems natural:
\begin{question}
 On which groups \(G\) and finite symmetric generating sets \(F\) does there exist a left invariant distance such that
 \[\sum\limits_{x \in F}\xi(x) < 0\]
 for all boundary horofunctions \(\xi\).
\end{question}

The condition seems related to hyperbolicity.  In fact one has the following:
\begin{remark}\label{rmk}
 Let \(G\) be a group and \(F\) a finite symmetric generating set.   The condition above is satisfied if and only if 
 \[\sum\limits_{x \in F}(\xi|x)_o  \le \frac{1}{2}\sum\limits_{x \in F}\dist(o,x)\]
 where \((x,y)_o\) is the Gromov product on \(G\) with base point $o$.
\end{remark}

\subsubsection{Hyperbolic Bernoulli percolation clusters}

Given natural numbers \(P,Q\) with \(\frac{1}{P} + \frac{1}{Q} < \frac{1}{2}\) we consider the regular tiling of the hyperbolic plane by regular \(P\)-gons with \(Q\) at each vertex.  Let \(o\) be a fixed point in the hyperbolic plane which will be a vertex of the tiling for all \(P,Q\).

For each \(P,Q\) the tiling can be identified with the Cayley graph of the cocompact Fuchsian group \(G\) generated by the set \(F\) of central symmetries with respect to the midpoints of edges incident to \(o\).  On \(G\) we consider the left invariant distance which comes from the hyperbolic distance via this identification.

\begin{figure}[h!]
\centering
\includegraphics[scale=0.45]{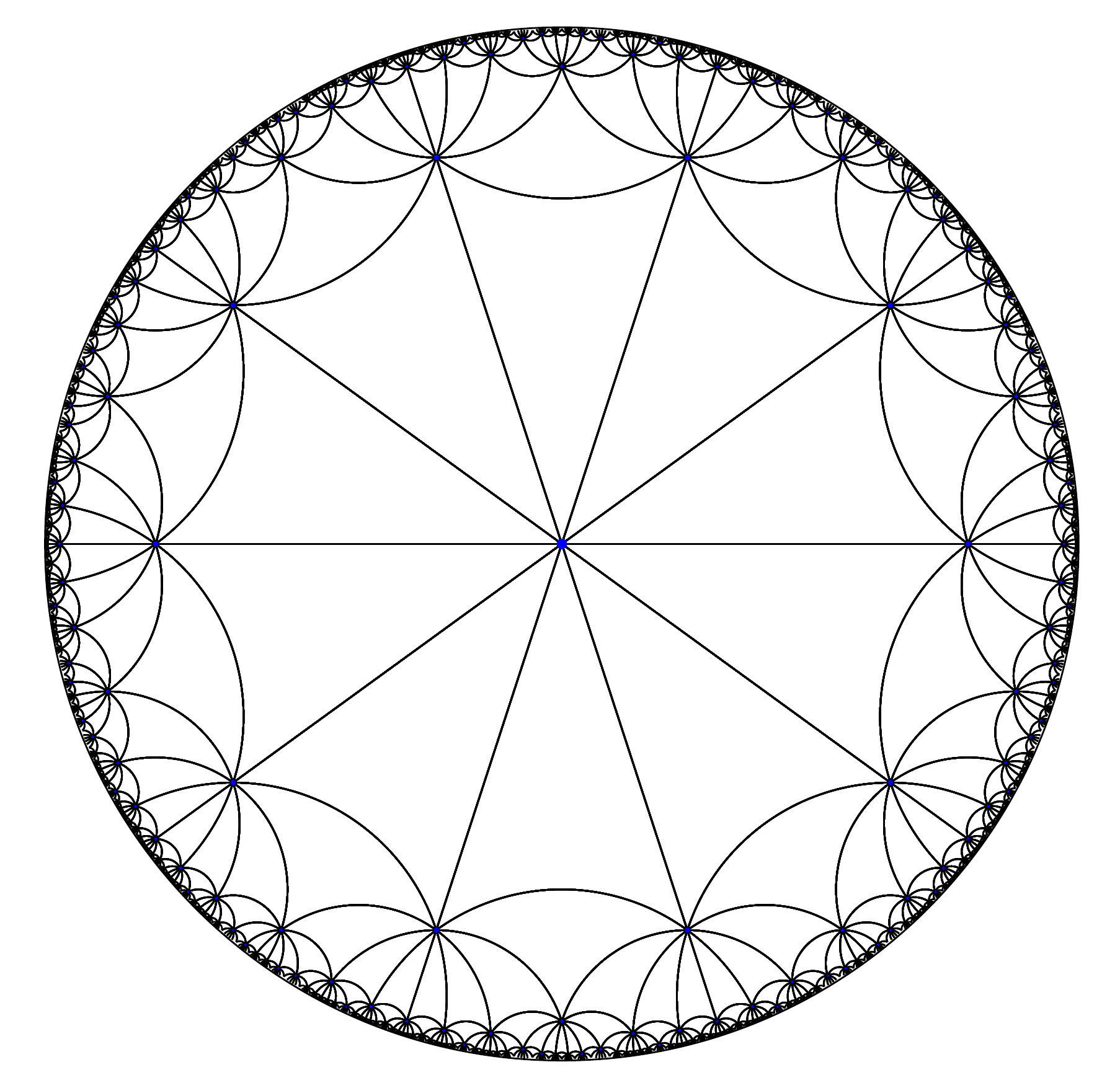}
\caption{\label{tessellation3-10figure}A tesselation by regular \(P\)-gons with \(Q\)-meeting at each vertex in the Poincaré disk model (here \(P=3\) and \(Q=10\)).}
\end{figure}

Consider a Bernoulli edge percolation \(G_p\) as in Theorem \ref{bernoullipercolationtheorem}.  And let \(\ell_p(P,Q)\) be the speed of the simple random walk on \(G_p\) conditioned on \(G_p\) being infinite.   From Theorem \ref{bernoullipercolationtheorem}, and the Lemmas \ref{hyperboliclemma1} and \ref{hyperboliclemma2} below, we obtain:
\begin{theorem}[Speed on percolation clusters of a hyperbolic tiling]\label{tilingtheorem}
In the above context \(\ell_p(P,Q) \ge 2\log(Q) - \frac{1}{p}O(\log(\log(Q)))\) when \(Q \to +\infty\), where the right hand side is independent of \(P\).
\end{theorem}

To establish the above result we need only to estimate the distance \(r(P,Q)\) between neighboring vertices of the tiling, and the maximum possible sum of \(f_\xi\) over all neighbors of \(o\).
\begin{lemma}\label{hyperboliclemma1}
Uniformly in \(P\) one has \(r(P,Q) = 2\log(Q) + O(1)\) when \(Q \to +\infty\).
\end{lemma}
\begin{proof}
Consider a triangle joining a vertex, the center, and the midpoint of a side, of a regular \(P\)-gon with interior angle \(2\pi/Q\).  The interior angles of this triangle are \(\pi/P,\pi/Q\) and \(\pi/2\), and the side opposite to the angle \(\pi/P\) has length \(r(P,Q)/2\).  

By the hyperbolic law of cosines one obtains
\[r(P,Q) = 2\acosh\left(\frac{\cos(\pi/P)}{\sin(\pi/Q)}\right).\]

We set \(r(\infty,Q) = \lim\limits_{P \to +\infty}r(P,Q)\).  Since \(r(\infty,Q)-r(3,Q)\) is uniformly bounded one obtains
\[r(P,Q) = 2\log(Q) + O(1)\]
uniformly in \(P\) when \(Q \to +\infty\).
\end{proof}

\begin{lemma}\label{hyperboliclemma2}
In the context above one has
\[\max\limits_{\xi} \frac{1}{Q}\sum\limits_{i = 1}^Q f_\xi(x_i) = O(\log(\log(Q)))\]
when \(Q \to +\infty\).
\end{lemma}
\begin{proof}
By Lemma \ref{conesetlemma} the set of points where \(f_\xi\) is larger than \(2\log(\log(Q))\) is contained in a cone with angle \(C\log(Q)^{-1}\) for some constant \(C\) independent of \(Q\).  Hence, there are at most \(O(Q/\log(Q))\) neighbors of \(o\) in this set.  Bounding the value of \(f_\xi\) at those points by \(2r_{P,Q} = 4\log(Q)+O(1)\) one obtains
\begin{align*}
\max\limits_{\xi} \frac{1}{Q}\sum\limits_{i = 1}^Q f_\xi(x_i) &\le \frac{1}{Q}O(Q/\log(Q))(4\log(Q)+O(1)) + 2\log(\log(Q))\\ &= O(\log(\log(Q)))
\end{align*}
which establishes the lemma.
\end{proof}

\section{Application to dimension drop of harmonic measures}

In this section we will show that the simple random walk on the tilings of the hyperbolic plane considered in Theorem \ref{tilingtheorem} has a harmonic measure which is singular with respect to the rotationally invariant measure on the boundary.

This was first proved in \cite{furstenbergformulaoldversion} as far as the author's are aware, and was re-obtained independently in recent work of Petr Kosenko \cite{kosenko}.   His result covers all but finitely many cases for the number of sides and polygons per vertex \((P,Q)\), while we estimate the dimension of the boundary measure only for large \(Q\).

Recall that \(o\) is a fixed point in the Hyperbolic plane which is a vertex of a tiling by regular \(P\)-gons with meeting \(Q\) meeting at each vertex.    Let \(x_0=o,x_1,\ldots\) be the simple random walk starting at \(o\) on this tiling.   One can write this as \(x_n = g_1\cdots g_n(o)\) for an i.i.d. sequence \(g_1,g_2,\ldots\) in the set \(F\) of central symmetries with respect to the midpoints of sides incident to \(o\).

In this context it is well known (see \cite{kosenko} and the references therein) that a limit point \(x_\infty = \lim\limits_{n \to +\infty} x_n\) on the boundary of the hyperbolic plane exists.  Let \(\nu_{P,Q}\) be the distribution of \(x_{\infty}\) and let \(\dim(\nu_{P,Q})\) be the Hausdorff dimension of \(\nu_{P,Q}\), that is, the infimum of the dimensions of full measure Borel sets.

Recall that the linear drift or speed is given by
\[\ell(P,Q) = \lim\limits_{n \to +\infty}\frac{1}{n}\dist(x_0,x_n)\]
while the asymptotic entropy (Avez entropy) is given by
\[h(P,Q) = \lim\limits_{n \to +\infty}-\frac{1}{n}\log(p_n(x_0,x_n))\]
where \(p_n(x,y)\) is probability of arriving in n steps at \(y\) starting from \(x\).

From \cite{tanaka} these quantities are related by
\[\dim(\nu_{P,Q}) = \frac{h(P,Q)}{\ell(P,Q)}.\]

\begin{theorem}[Dimension drop for some cocompact Fuchsian groups]\label{dimdrop}
In the context above one has \(\limsup\limits_{Q \to +\infty}\dim(\nu_{P,Q}) \le \frac{1}{2}\) uniformly in \(Q\).
\end{theorem}
\begin{proof}
By subadditivity one has \(h(P,Q) \le -\E{\log(p_1(x_0,x_1))} = \log(Q)\).

Combining this with the estimate for \(\ell(P,Q)\) given by Theorem \ref{tilingtheorem} one obtains
 \[\dim(\nu_{P,Q}) = \frac{h(P,Q)}{\ell(P,Q)} \le \frac{\log(Q)}{2\log(Q) + O(\log(\log(Q)))},\]
 from which the result follows immediately.
\end{proof}

\section{Application to cocycles of isometries}

We will now show that Theorem \ref{furstenbergtheorem} allows one to reobtain the law of large numbers from \cite{karlsson-ledrappier}.    This result has a wealth of applications such as Oseledet's theorem, and geodesic tracking theorems.   We will also illustrate that the well known formula from \cite{furstenberg1963b} for the first Lyapunov exponent of a product of \(2\times 2\) i.i.d.\ matrices also follows from Theorem \ref{furstenbergtheorem}, whereas it does not follow from \cite[Theorem 1.1]{karlsson-ledrappier}.

\subsection{Law of large numbers}

Let \((g_n)_{n \in \Z}\) be a stationary and ergodic random sequence of isometries of a complete separable metric space \((X,\dist)\).   Let \(x_0 = o\) be  a fixed point in \(X\) and for \(n = 1,2,\ldots\) let 
\(x_n = g_n\cdots g_1(o)\) and \(x_{-n} = g_{-n+1} \cdots g_0^{-1}(o)\).

The sequence \((x_n)_{n \in \Z}\) is distance stationary and its speed
\[\ell = \lim\limits_{n \to +\infty}\frac{1}{n}\dist(x_0,x_n)\]
is deterministic because it is a shift invariant function of the sequence \((g_n)_{n \in \Z}\).

\begin{theorem}[Karlsson-Ledrappier law of large numbers]\label{karlssonledrappiertheorem}
In the context above almost surely there exists a horofunction \(\xi\) such that
\[\ell = \lim\limits_{n \to +\infty}\frac{1}{n}\xi(x_n)\]
\end{theorem}
\begin{proof}
Consider the reverse sequence \(y_n = x_{-n}\).  Notice that the speed of this sequence is also \(\ell\) since by distance stationarity and ergodicity
\[\ell = \lim\limits_{n \to +\infty}\E{\frac{1}{n}\dist(x_0,x_n)} = \lim\limits_{n \to +\infty}\E{\frac{1}{n}\dist(x_{-n},x_0)}.\]

By Theorem \ref{furstenbergtheorem} there exists a random horofunction \(\xi\) such that
\[\ldots, \xi(y_1) - \xi(y_0), \ldots, \xi(y_n)-\xi(y_{n+1}), \ldots\]
is stationary and has Birkhoff limit \(\ell\).

In particular one has
\begin{align*}
\ell &= \lim\limits_{n \to +\infty}\frac{1}{n}\sum\limits_{k = 1}{n}\xi(y_{-k}) - \xi(y_{-k+1})\\ 
& = \lim\limits_{n \to +\infty}\frac{1}{n}(\xi(y_{-n})-\xi(y_0)) = \lim\limits_{n \to +\infty}\frac{1}{n}\xi(x_{n}).
\end{align*}
\end{proof}

\begin{remark}
The sign convention for horofunctions we use here is the opposite as in \cite[Theorem 1.1]{karlsson-ledrappier}.
\end{remark}

\begin{remark}
While in \cite[Theorem 1.1]{karlsson-ledrappier} the horofunction \(\xi\) is a measurable function of the sequence of isometries \((g_n)_{n \in \Z}\), in Theorem \ref{karlssonledrappiertheorem} above it depends also on an auxiliary random variable \(u\).
\end{remark}

A natural question is why one has to reverse the sequence to obtain Theorem \ref{karlssonledrappiertheorem} from Theorem \ref{furstenbergtheorem}.  A preliminary answer is that, in many applications, the horofunction obtained from the past tail of a distance stationary sequence has some independence from the first step of the sequence.  This makes the formula \(\ell = -\E{\xi(x_1)}\) more powerful and easier to estimate.

In particular, when the sequence of isometries \((g_n)_{n \in \Z}\) is i.i.d.\ the horofunction given by Theorem \ref{furstenbergtheorem} is independent from \(x_1\).  This is illustrated in the following subsection.

\subsection{Lyapunov exponents of $2\times 2$ i.i.d.\ matrix products}

Suppose that $(A_n)_{n \in \Z}$ is an i.i.d.\ sequence of matrices in $\text{SL}(2,\R)$ with the additional property that $\E{\log(|A_1|)} < +\infty$ where $|A|$ denotes the operator norm of the matrix $A$.

The largest Lyapunov exponent of the sequence $(A_n)$ is defined by
\[\chi = \lim\limits_{n \to +\infty}\frac{1}{n}\log\left(|A_n\cdots A_1|\right),\]
and is almost surely constant since it is a tail function of the sequence.

Notice that if one writes $A_n \cdots A_1 = O_n P_n$ where $O_n$ is orthogonal and $P_n$ is symmetric with positive eigenvalues, one obtains
\[\chi = \lim\limits_{n \to +\infty}\frac{1}{n}\log\left(|P_n|\right).\]

This implies that $\xi$ depends only on the sequence of projections of $A_n\cdots A_1$ to the left quotient $M = \text{SO}(2)\backslash\text{SL}(2,\R)$.  Let $[A]$ denote the equivalence class of a matrix $A \in \text{SL}(2,\R)$ in the quotient above.

The quotient space admits a (unique up to homotethy) Riemannian metric for which the transformations $[A] \mapsto [AB]$ are isometries for all $B \in \text{SL}(2,\R)$.  One may choose such a metric so that the distance $d([\text{Id}],[A]) = \sqrt{\log(\sigma_1)^2 + \log(\sigma_2)^2}$ where $\sigma_1,\sigma_2$ are the singular values of $A$ and $\text{Id}$ denotes the identity matrix.  In particular, since $\sigma_1 = |A|$ and $A$ has determinant $1$, one obtains $d(\text{Id},[A]) = \sqrt{2}\log(|A|)$.

With the Riemannian metric under consideration the sequence 
\[\ldots, x_{-2} = [A_{-1}^{-1}A_0^{-1}], x_{-1} = [A_0^{-1}], x_0 = [\text{Id}], x_1 =[A_1], x_2 = [A_2A_1],\ldots\]
is distance stationary and satisfies $\E{d(x_0,x_1)} < +\infty$.  Furthermore, its rate of escape is $\ell = \sqrt{2}\chi$.

The boundary horofunctions on $M$ are of the form $\xi([A]) = -\sqrt{2}\log(|Av|)$ for some $|v| = 1$ (see for example \cite{MR1794514}).

If $A_1$ is not contained almost surely in a compact subgroup of $\text{SL}(2,\R)$ then one may use \cite[Theorem 8]{MR0423532} to show that $\P(x_n \in K) \to 0$ when $n \to +\infty$ for all compact sets $K \subset M$.

Hence, Theorem \ref{furstenbergtheorem} and Proposition \ref{corollaryboundary} imply the existence of a random unit vector $v \in \R^2$ which is independent from $A_1$ and such that
\[\chi = \E{\log\left(|A_1v|\right)}.\]

In particular, letting $\mu$ be the distribution of $A_1$, there is a probability $\nu$ on the unit circle $S^1 \subset \R^2$ such that 
\[\chi = \int\limits_{\text{SL}(2,\R)} \int\limits_{S^1} \log\left(|Av|\right) \mathrm{d}\nu(v) \mathrm{d}\mu(A).\]

This is Furstenberg's formula for the largest Lyapunov exponent (see \cite[Theorem 3.6]{bougerol-lacroix}).

The reasoning above may be carried out in $\text{SL}(n,\R)$ for larger $n$.  What results is a formula for the sum of squares of the Lyapunov exponents of the random i.i.d.\ product of matrices.  As above, the distribution of the random boundary horofunction is unknown (in larger dimension horofunctions are determined by a choice of a flag and a sequence of weights adding up to zero).

\section{Proof of the formula for speed}\label{sec:furstenberg}

\subsection{Preliminaries}

\subsubsection{Distance stationary sequences}

In what follows $(M,d)$ denotes a complete separable metric space and $o \in M$ a base point which is fixed from now on. Recall that $M$ is \emph{proper} if closed balls are compact, and that any proper metric space is separable. If $M$ is a geodesic metric space, then $M$ is proper if and only if it is complete and locally compact.

A random sequence $(x_n)_{n \in \Z}$ of points in $M$ is said to be \emph{distance stationary} if the distribution of $(d(x_m,x_n))_{m,n \in \Z}$ coincides with that of $(d(x_{m+1},x_{n+1}))_{m,n \in \Z}$.

\subsubsection{Horofunctions}

To each point $x \in M$ we associate a horofunction $\xi_x: M \to \R$ defined by
\[\xi_x(y) = d(x,o) - d(x,y).\]
Notice that $\xi_x$ is normalized so that $\xi_x(o)=0$. 

The mapping $x \mapsto \xi_x$ has image in the space of $1$-Lipschitz functions from $M \to \R$.  If $x, x' \in M$ and $d(x, o) \geq d(x', o)$, then
$$\xi_x(x)-\xi_{x'}(x)= d(x, o)-d(x',o)+d(x',x)\geq d(x',x),$$
and hence $x\mapsto \xi_x$ is injective. In what follows we identify $M$ with the image $\left\{ \xi_x : x \in M \right\}$ as a subset of the space of 1-Lipschitz functions.

We equip this space with the topology of uniform convergence on compact sets. Since $\|\xi_x-\xi_{x'}\|_\infty\leq 2d(x,x')$, the mapping $x\mapsto \xi_x$ is continuous. The \emph{horofunction compactification} of $M$ is the space $\widehat{M}$ obtained as the closure of $M$ via this identification.

Since $M$ is separable, from the $1$--Lipschitzness of $\left\{ \xi_x  : x \in M \right\}$, uniform convergence on compact sets of $M$ is equivalent to the topology of pointwise convergence on any countable dense set $\left\{ q_1,q_2,\dots \right\} \subset M.$  As $|\xi_x(y)| \leq d(o,y)$ for any $y,$ it follows from diagonalization that $\left\{ \xi_x  : x \in M \right\}$ is precompact.  Moreover, if we define a metric
$$\rho(\xi_x, \xi_y) = \sum_{j=1}^\infty 2^{-j}\min\{ |\xi_x(q_j) - \xi_y(q_j)|, 1\},$$
the completion $\widehat M$ of $M$ with respect to $\rho$ will be a compact, complete separable metric space. In particular, $\widehat{M}$ is a Polish space.

If $M$ is proper, then uniform convergence on compact sets is equivalent to uniform convergence on bounded sets, so this construction recovers the usual horofunction compactification of $M$ for proper separable metric spaces. See \cite[II.1]{Ballmann} for details.
 
\begin{proposition}\label{prop:embedding}
Let $M$ be a geodesic and proper metric space. Then for any bounded open set $U$ we have that the closure $\overline{M\setminus U}$ in $\widehat{M}$ is disjoint from $U$.

In particular, the mapping $x\mapsto \xi_x$ is an embedding.
\end{proposition}
\begin{proof}
Suppose $(x_n)$ is a sequence in $M\setminus U$ and that $\xi_x=\lim_n \xi_{x_n}$ for some $x\in U$. If $(x_n)$ is bounded, by properness we can extract a convergent subsequence $(y_m)$ of $(x_n)$ converging to a point $y\in M\setminus U$. Since $x\mapsto \xi_x$ is continuous, $\xi_{y_n}\to \xi_{y}$, but by injectivity we must have that $y=x\in U$ which is a contradiction.

Let $r>0$ be such that $U\subset B(o,r)$. We assume without loss of generality that $d(x_n,o)\geq 2r+1$. We have in particular that $d(x_n,o)\geq 2d(x,o)+1$. Since $d(x,x_n)\geq d(x,o)+1$, we can pick a point $y_n$ in a geodesic segment $[x,x_n]$ so that $d(x,y_n)=d(x,o)+1$.

For this point $y_n$ we have
\begin{align*}
\xi_{x_n}(y_n)-\xi_x(y_n)&=d(x_n,o)-d(x_n,y_n)-d(x,o)+d(x,y_n)\\
&= d(x_n,o)-d(x_n,y_n) +1\\
&\geq d(x_n,x)-d(x,o)-d(x_n,y_n)+1\\
&=d(x,y_n)-d(x,o)+1=2
\end{align*}
If we let $K=\{y_n:n\geq 1\}$, then $\sup_{z\in K}|\xi_x(z)-\xi_{x_n}(z)|\geq 2$. Since $K$ is compact, we have a contradiction.
\end{proof}

%Compactness of $\widehat{M}$ follows from the Arsel\`a-Ascoli theorem and the fact that all functions $\xi_x$ are $1$-Lipschitz.  A horofunction on $M$ is an element of $\widehat{M}$.

Horofunctions in $\widehat{M}$ which are not of the form $\xi_x$ will be called \emph{boundary horofunctions} and the set of boundary horofunctions is the horofunction boundary of $M$, which might sometimes be written $\widehat{M} \setminus M$ abusing notation slightly. As a corollary of Proposition \ref{prop:embedding} notice that when $M$ is geodesic and proper, if $(x_n)$ is a sequence escaping any bounded set, then every limit point of $(\xi_{x_n})$ is a boundary horofunction.

\subsubsection{Speed or linear drift}

If $(x_n)_{n \in \Z}$ is a distance stationary sequence in $M$ satisfying $\E{d(x_0,x_1)} < +\infty$, then by Kingman's subadditive ergodic theorem the random limit
\[\ell = \lim\limits_{n \to +\infty}\frac{1}{n}d(x_0,x_n)\]
exists almost surely and in $L^1$.

We call this limit the \emph{speed} or \emph{linear drift} of $(x_n)_{n \in \Z}$.

\subsubsection{Stationary sequences and Birkhoff limits}

Recall that a random sequence $(s_n)_{n \in \Z}$ is stationary if its distribution coincides with that of $(s_{n+1})_{n \in \Z}$.  

If $(s_n)_{n \in \Z}$ is stationary and $\E{|s_1|} < +\infty$, then by Birkhoff's ergodic theorem the limit 
\[\lim\limits_{n \to +\infty}\frac{1}{n}\sum\limits_{i = 0}^{n-1}s_i\]
exists almost surely and in $L^1$.  We call this limit the Birkhoff limit of the sequence.

\subsubsection{Spaces of probability measures and representations}

Given a Polish space $X$ we use $\mathcal{P}(X)$ to denote the space of Borel probability measures on $X$ endowed with the topology of weak convergence (i.e. a sequence converges if the integral of each continuous bounded function from $X$ to $\R$ does).   This space is also Polish and is compact if $X$ is compact.

We will use the following result due to Blackwell and Dubins (see \cite{blackwell-dubins1983}):
\begin{theorem}[Continuous representation of probability measures]\label{blackwell-dubinstheorem}
For any Polish space $X$ there exists a function $F:\mathcal{P}(X)\times [0,1] \to X$ such that if $u$ is a uniform random variable on $[0,1]$ the following holds:
\begin{enumerate}
 \item For each $\mu \in \mathcal{P}(X)$ The random variable $F(\mu,u)$ has distribution $\mu$.
 \item If $\mu_n \to \mu$ then $F(\mu_n,u) \to F(\mu,u)$ almost surely.
\end{enumerate}
\end{theorem}

We call a function $F$ satisfying the properties in the above theorem a continuous representation of $\mathcal{P}(X)$.

\subsubsection{Application of Koml\'os' theorem to random probabilities}

Recall that a sequence $(a_n)_{n \ge 1}$ is Cesaro convergent if the limit $\lim\limits_{n \to +\infty}\frac{1}{n}\sum\limits_{k = 1}^n a_k$ exists.  We restate the main result of \cite{komlos1967}.

\begin{theorem}[Koml\'os' theorem]
  Let $(X_n)_{n \ge 1}$ be a sequence of real valued random variables with $\sup\limits_n\E{|X_n|} < +\infty$.  Then there exists a deterministic subsequence $\{n_k\}_1^\infty$ and a random variable $Y$ with finite expectation so that the sequence $Y_{j} = X_{n_j}$ Cesaro converges almost surely to $Y$.  Furthermore, for any deterministic increasing sequence $\{a_j\}_1^\infty$, $(Y_{a_j})_{j \ge 1}$ has the same property.
\end{theorem}

We will need the following corollary of Koml\'os' theorem.
\begin{corollary}\label{komloscorollary}
  Let $(\mu_n)_{n \ge 0}$ be a sequence of random probabilities on a compact metric space $(X,d)$.  There exists a deterministic sequence $\{ n_j\}_1^\infty$ and a random probability $\mu$ on $X$ so that the subsequence $(\mu_{n_j})_{j \ge 1}$ Cesaro converges almost surely to $\mu$ on $X$.
\end{corollary}
\begin{proof}
In this proof we use the notation $\nu(f) = \int\limits_{X} f(x)d \nu(x)$.

Let $(f_n)_{n \ge 1}$ be a dense sequence in the space of continuous functions from $X$ to $\R$ (with respect to the topology of uniform convergence).

 Applying Koml\'os' theorem to $\left(\mu_n(f_1)\right)_{n \ge 1}$ one obtains a subsequence $n_{1,k} \to +\infty$ such that $\mu_{n_{1,k}}(f)$ Cesaro converges almost surely and any further subsequence has the same property.
 
 For $i = 1,2,3,4,\ldots$, inductively applying Koml\'os' theorem to $(\mu_{n_{i,k}}(f_{i+1}))_{k \ge 1}$ we obtain a subsequence $(n_{i+1,k})_{k \ge 1}$ of $(n_{i,k})_{k \ge 1}$ such that $\mu_{n_{i+1,k}}(f_{i+1})$ Cesaro converges almost surely and any further subsequence has the same property.
 
 Setting $n_k = n_{k,k}$ one obtains that $(\mu_{n_k})_{n \ge 1}$ Cesaro converges to a random probability $\mu$ almost surely.
\end{proof}

\subsection{Furstenberg type formula for speed}

\subsubsection{Statement and proof}

\begin{theorem}[Furstenberg type formula for distance stationary sequences]\label{furstenbergtheorem}
Let $(x_n)_{n \in \Z}$ be a distance stationary sequence in a complete separable metric space $(M,d)$ satisfying $\E{d(x_0,x_n)} < +\infty$ and $\ell$ be its linear drift.

Suppose there exists a random variable $u$ which is uniformly distributed on $[0,1]$ and independent from $(x_n)_{n \in \Z}$.  Then the following holds:
\begin{enumerate}
 \item The sequence of random probability measures on $\widehat{M}$ defined by $\mu_n = \frac{1}{n}\sum\limits_{i = 1}^n \delta_{\xi_{x_{-i}}}$ has a subsequence which is almost surely Cesaro convergent to a random probability $\mu$.
 \item There exists a random horofunction $\xi$ which is measurable with respect to $\sigma(u,\mu)$ and whose conditional distribution given $(x_n)_{n \in \Z}$ is $\mu$.
 \item The sequence of increments $(\xi(x_n)-\xi(x_{n+1}))_{n \in \Z}$ is stationary and its Birkhoff limit equals $\ell$ almost surely.  In particular, $\E{\ell} = \E{\xi(x_0)-\xi(x_1)}$.
\end{enumerate}
\end{theorem}
\begin{proof}
The fact that $(\mu_n)_{n \ge 1}$ has an almost surely Cesaro convergent subsequence follows directly from the version of Koml\'os' theorem for random probabilities given above (see Corollary \ref{komloscorollary}).  Let $(\mu_{n_j})_{j \ge 1}$ be such a subsequence and $\mu$ be its almost sure Cesaro limit.

Let $F:\mathcal{P}(\widehat{M})\times [0,1] \to \widehat{M}$ be continuous representation of $\mathcal{P}(\widehat{M})$, as given by Theorem \ref{blackwell-dubinstheorem} and define $\xi = F(u,\mu)$.  Clearly $\xi$ is $\sigma(u,\mu)$-measurable and its conditional distribution given $(x_n)_{n \in \Z}$ is $\mu$.

We will now show that $\E{\xi(x_0)-\xi(x_1)} = \ell$.  

For this purpose let $\xi_k = F(u,\frac{1}{k}\sum\limits_{j = 1}^k\mu_{n_j})$ and notice that $\xi_k \to \xi$ almost surely when $k \to +\infty$.  Because horofunction are $1$-Lipschitz one has $|\xi_k(x_0)-\xi_k(x_1)| \le d(x_0,x_1)$.   Since $\E{d(x_0,x_1)} < +\infty$ this implies that the sequence is uniformly integrable and one obtains $\E{\xi(x_0)-\xi(x_1)} = \lim\limits_{k \to +\infty}\E{\xi_k(x_0) - \xi_k(x_1)}$.

For the sequence on the right hand side using distance stationarity one obtains

\begin{align*}
\E{\xi_k(x_0) - \xi_k(x_1)} &= \frac{1}{k}\sum\limits_{j = 1}^k\E{\frac{1}{n_j}\sum\limits_{i = 1}^{n_j} \xi_{x_{-i}}(x_0) - \xi_{x_{-i}}(x_1)}
\\ &= \frac{1}{k}\sum\limits_{j = 1}^k\E{\frac{1}{n_j}\sum\limits_{i = 1}^{n_j} -d(x_{-i},x_0) + d(x_{-i},x_1)}
\\ &= \frac{1}{k}\sum\limits_{j = 1}^k\E{\frac{1}{n_j}\sum\limits_{i = 1}^{n_j} -d(x_{0},x_i) + d(x_0,x_{i+1})}
\\ &= \frac{1}{k}\sum\limits_{j = 1}^k\E{\frac{-d(x_0,x_1) + d(x_0,x_{n_j+1})}{n_j}}.
\end{align*}

Taking the limit when $k \to +\infty$ above it follows that $\E{\xi(x_0)-\xi(x_1)} = \E{\ell}$ as claimed.

We will now prove that $\left(\xi(x_n) - \xi(x_{n+1})\right)_{n\in \Z}$ is stationary.  

Suppose $F:\R^\Z \to \R$ is continuous and bounded and notice that by distance stationarity one has
\begin{equation*}
\begin{split}
\mathbb{E}(F((\xi_k(x_{n+1}) &- \xi_k(x_{n+2}))_{n \in \Z}))\\ &= \frac{1}{k}\sum\limits_{j = 1}^k\frac{1}{n_j}\sum\limits_{i = 1}^{n_j}\E{ F\left(\left(\xi_{x_{-i}}(x_{n+1}) - \xi_{x_{-i}}(x_{n+2})\right)_{n \in \Z}\right)}
\\ &= \frac{1}{k}\sum\limits_{j = 1}^k\frac{1}{n_j}\sum\limits_{i = 1}^{n_j}\E{ F\left(\left(\xi_{x_{-(i-1)}}(x_{n}) - \xi_{x_{-(i-1)}}(x_{n+1})\right)_{n \in \Z}\right)}
\\ &= \frac{1}{k}\sum\limits_{j = 1}^k\frac{1}{n_j}\sum\limits_{i = 0}^{n_j-1}\E{ F\left(\left(\xi_{x_{-i}}(x_{n}) - \xi_{x_{-i}}(x_{n+1})\right)_{n \in \Z}\right)}
\\ &= C_k \max |F| + \E{F\left(\left(\xi_k(x_{n}) - \xi_k(x_{n+1})\right)_{n \in \Z}\right)}
\end{split}
\end{equation*}
where $|C_k| \le \frac{1}{k}\sum\limits_{j = 1}^k \frac{2}{n_j} \to 0$ when $k \to +\infty$.

From this the stationarity of the increments of $\xi$ along the sequence $(x_n)_{n \in \Z}$, follows directly taking limit when $k \to +\infty$.

By Birkhoff's theorem the Birkhoff averages of the increments of $\xi$ along the sequence exist almost surely and in $L^1$.  Additionally, because horofunctions are $1$-Lipschitz, one has
\[\lim\limits_{n \to +\infty} \frac{1}{n}\sum\limits_{k = 0}^{n-1} \xi(x_k) - \xi(x_{k+1}) \le \ell\]
almost surely.  But the expectation of the left hand side in the above inequality is $\E{\xi(x_0)-\xi(x_1)} = \E{\ell}$.  Hence both sides coincide almost surely.  This concludes the proof.
\end{proof}

\subsubsection{Boundary horofunctions}

The question of whether the random horofunction $\xi$ given by Theorem \ref{furstenbergtheorem} is almost surely on the horofunction boundary of $M$ sometimes arises.

A trivial example where this is not the case is obtained by letting $(x_n)_{n \in \Z}$ be an i.i.d. sequence of uniformly distributed random variables on $[0,1]$.  In this case the horofunction $\xi$ given by Theorem \ref{furstenbergtheorem} will be uniformly distributed on $[0,1]$ and independent from the sequence.

In the previous example the linear drift $\ell$ was $0$ almost surely.  It is not difficult to show that if $\ell > 0$ almost surely then $\xi$ must be a boundary horofunction almost surely.

However, in many examples Theorem \ref{furstenbergtheorem} can be used to decide whether or not $\ell$ is positive.  Hence it is useful to have a criteria for establishing that $\xi$ is almost surely on the horofunction boundary without knowledge of $\ell$.   The following proposition is such a result.

\begin{proposition}\label{corollaryboundary}
In the context of Theorem \ref{furstenbergtheorem}, if in addition $M$ is geodesic and proper, and if $\Pof{x_n \in K} \to 0$ when $n \to -\infty$ for all bounded sets $K$, then the random horofunction $\xi$ is almost surely on the horofunction boundary.
\end{proposition}
\begin{proof}
We will use the notation from Theorem \ref{furstenbergtheorem}.  Let $\nu_n$ denote the sequence of averages of the subsequence of the probabilities $(\mu_n)$ which Cesaro converges to $\mu$ almost surely.

Given a bounded set $K$ pick an open ball $U=B(o,r)$ containing the closure of $K$. From the hypothesis it follows that 
$$\E{\mu_n(U)}=\E{\frac{1}{n}\sum_{i=1}^n\mathds{1}(x_{-i}\in U)}=\frac{1}{n}\sum_{i=1}^n \Pof{x_{-i}\in U} \to 0$$
when $n \to +\infty$.   Therefore $\E{\nu_n(U)} \to 0$ as well. 

By Proposition \ref{prop:embedding} the mapping $x\mapsto \xi_x$ is an embedding, so $U$ is open in $\widehat{M}$. Because $\nu_n \to \mu$ one has $\mu(U) \le \liminf\limits_n \nu_n(U)$ almost surely.  Combining this with Fatou's lemma one obtains
\[\Pof{\xi \in K} = \E{\mu(K)} \le \E{\mu(U)} \le \E{\liminf\limits_n\nu_n(U)} \le \liminf\limits_n \E{\nu_n(U)} = 0,\]
and hence $\xi \notin K$ almost surely.
\end{proof}

\appendix

\section{The sum of a horofunction and the distance function in the hyperbolic plane}

Recall that given a boundary horofunction $\xi$ we have defined $f_\xi(x) = \xi(x) + d(o,x)$.  

We recall that the upper half plane model of the hyperbolic plane is obtained identifying $\H^2$ with $\lbrace x+iy \in \C: y > 0\rbrace$ with the metric $\frac{1}{y^2}(dx^2+dy^2)$.   In this model we will set the base point $o = i$.  We will use explicit formulas for the distance function and horofunctions in this model, as well as the correspondence between the horofunction boundary and points on the extended real line (see for example \cite[Exercises 2.2, 6.10, 6.11]{bonahon}).

\begin{lemma}\label{ellipselemma}
In the upper half plane model of the hyperbolic plane the function $f_\xi$ asociated to the boundary point at $\infty$ is given by  
\[f(z) = 2\log\left(\frac{|z-i| + |z+i|}{2}\right).\]

In particular $f$ extends to all of $\C$ as a continuous function whose level sets are ellipses with foci at $\pm i$.
\end{lemma}
\begin{proof}
  The proof is by direct calculation.   The horofunction asociated to the boundary point at $\infty$ is $\xi(x+iy) = \log(y)$.  The distance $d(o,x+iy)$ can be calculated explicitely and is given by
  \[d(o,x+iy) = 2\log\left(\frac{|z-i|+|z+i|}{2\sqrt{y}}\right).\]
\end{proof}

\begin{figure}[h!]
\centering
\includegraphics[scale=1.2]{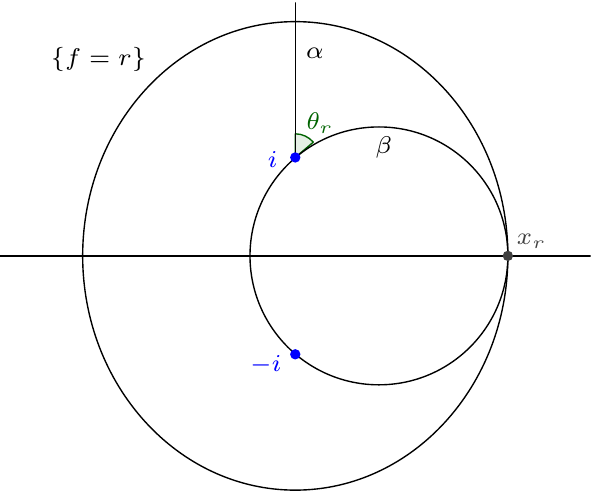}
\caption{\label{conesetfigure}Illustration of the proof of Lemma \ref{conesetlemma}.}
\end{figure}

The above calculation allows us to estimate the angular size of the level sets of $f_\xi$ as viewed from $o$.
\begin{lemma}\label{conesetlemma}
 For each $r > 0$ and all boundary horofunctions $\xi$ on $\H^2$ the set $\lbrace f_\xi > r\rbrace$ occupies a visual angle of at most \(4\arctan((e^{r}-1)^{-1/2})\) when viewed from \(o\).
\end{lemma}
\begin{proof}
 Given $\xi$ there is a unique unit speed geodesic $\alpha(t) = \exp_o(tv)$ such that $\xi(\alpha(t)) = t$.

 In the upper half plane model of $\H^2$ we may assume that the geodesic $\alpha(t)$ discussed above is $\alpha(t) = e^t i$.  And therefore that $\xi(x+iy) = \log(y)$.  By Lemma \ref{ellipselemma} one has 
 \[f_\xi(z) = f(z) = 2\log\left(\frac{|z-i|+|z+i|}{2}\right).\]
 
 Let $x_r > 0$ be such that $f(x_r) = r$.   It suffices to calculate the angle $\theta_r$ at $o$ between the geodesic ray $\alpha$ and the geodesic ray $\beta$ starting at $o$ whose endpoint is $x_r$.
 
 For this purpose we use the conformal transformation $z \mapsto \frac{z-i}{z+i}$ which maps the upper half plane to the unit disk.  Notice that $\alpha$ goes to the segment $[0,1]$ under this transformation.  On the other hand $\beta$ goes to another radius of the unit disk.  Hence, the angle $\theta_r$ is the absolute value of the smallest argument of $\frac{x_r-i}{x_r+i}$ from which one obtains
 \[\theta_r = 2\arctan\left(\frac{1}{x_r}\right).\]
 
 To conclude the proof one calculates from the equation $f(x_r) = r$ obtaining
 \[x_r = \sqrt{e^r-1}.\]
\end{proof}


\end{document}